\documentclass[11pt]{article} 

\usepackage[utf8]{inputenc} 

\usepackage{geometry} 
\usepackage{graphicx} 

\usepackage{booktabs} 
\usepackage{array} 
\usepackage{paralist} 
\usepackage{verbatim} 
\usepackage{subfig} 

\usepackage{fancyhdr} 
\usepackage{sectsty}
\allsectionsfont{\sffamily\mdseries\upshape} 

\usepackage[nottoc,notlof,notlot]{tocbibind} 
\usepackage[titles,subfigure]{tocloft} 

\usepackage{algorithm}
\usepackage[noend]{algpseudocode}

\usepackage{amssymb}
\usepackage{amsmath,amsthm}
\usepackage{amsfonts}

\usepackage{tikz,pgfplots,pgf}
\usepackage{neuralnetwork}
\usepackage{tikz}

\newtheorem{theorem}{Theorem}[section]

\newtheorem{proposition}[theorem]{Proposition}
\newtheorem{definition}[theorem]{Definition}
\newtheorem{corollary}{Corollary}

\newtheorem{remark}[theorem]{Remark}

\begin{document}

\begin{center}
\textbf{\Large On a Class of Non-linear Differential Equations Arising from Branching Diffusions}
 \end{center}
\begin{center}
Erfan Salavati \footnote{Email: erfan.salavati@aut.ac.ir}\\
\end{center}

\begin{center}

\emph{\footnotesize Department of Mathematics and Computer Science, Amirkabir University of Technology}

\end{center}

\begin{abstract}
A non-linear differential equation arising from a stochastic process known as branching Brownian motion is considered. We find an explicit solution and show the uniqueness of the solution under some boundedness conditions using probabilistic ideas. We discuss non-negative solutions. We also generalize this idea to a class of non-linear parabolic differential equations which we call probabilistic parabolic equations.
\end{abstract}



\section{Introduction}

It is well-known that probabilistic methods are very useful in the study of linear partial differential equations (see~\cite{Oksendal}). But their usefulness in the study of non-linear partial differential equations have less been noticed in the literature. As diffusion processes can be thought as the microscopic level of linear parabolic equations, branching diffusion processes can be thought as the microscopic level of some linear parabolic equations. This gives us probabilistic interpretations for some complicated non-linear PDEs (see~\cite{Dynkin}).

In this article we wish to show that how one can use branching diffusion processes to study uniqueness properties of a class of non-linear PDEs. In section~\ref{section: equation} we study a specific parabolic PDE on $\mathbb{R}$ and introduce the main ideas of our method. In section~\ref{section: uniqueness} we prove one of our main results on uniqueness of the solution of the equation introduced in section~\ref{section: equation}. In section~\ref{section: non-negative} we discuss and prove a result on non-negative solutions. In section~\ref{section: general} we generalise the ideas of the previous sections and prove the same results for a more general class of PDEs which we call probabilistic parabolic equations.

\section{A Non-linear Differential Equation}\label{section: equation}

Let a particle be at the origin on the real line and assume it performs a standard Brownian motion until a random time which is exponentially distributed with parameter $\mu$ and at that time with probability $\frac{1}{2}$ the particle dies and with probability $\frac{1}{2}$ it gives birth to two particles same as itself. Let $X(t)$ be the (multi-)set of particles at time $t$.

We wish to study probability that some particle eventually reaches an arbitrary point $x$. We consider the complement probability, that is the probability that no one of the particles reaches $x$. Call this probability $p(x)$. Also, denote by $r(x,t)$ the probability that the particles extinct and never reach the point $x$ and denote by $s(x,t)$ the probability that the particles never reach $x$.

It is clear that $r(x,t) \le p(x) \le s(x,t)$. As is shown in~\cite{Grimmett}, page 522, $s(x,t)$ satisfies the following PDE,
\begin{equation} \label{main_evolution_PDE}
\frac{\partial u}{\partial t} = \frac{1}{2}\frac{\partial^2 u}{\partial x^2}+\frac{\lambda (u-1)^2}{2},\quad t\ge 0,\quad x\ge 0
\end{equation}
with initial and boundary conditions
\[ s(0,t)=0,\quad s(x,0)=1 \]
and a similar argument shows that $r(x,t)$ satisfies~\eqref{main_evolution_PDE} with initial and boundary conditions
\[ r(0,t)=0,\quad r(x,0)=0 \]
and also $p(x)$ satisfies
\begin{equation} \label{main_PDE}
0 = \frac{1}{2}\frac{\partial^2 u}{\partial x^2}+\frac{\lambda (u-1)^2}{2},\quad x\ge 0
\end{equation}
with boundary condition
\[ p(0)=0\]

One can check easily that
\begin{equation} \label{explicit_solution}
u(x)=1-\frac{1}{(\sqrt{\frac{\lambda}{6}} x+1)^2}
\end{equation}
is a solution of~\eqref{main_PDE} with boundary condition $u(0)=0$ and satisfies $0\le u(x)\le 1$. In the next section, we wish to show that this solution is unique and hence $p(x)$ is given by~\eqref{explicit_solution}.

\section{Uniqueness of Solution}\label{section: uniqueness}
In this section we use~\eqref{main_evolution_PDE} to prove the uniqueness of the solution of~\ref{main_PDE}. It is obvious that any solution of~\eqref{main_PDE} is a stationary solution of~\eqref{main_evolution_PDE}.

\begin{theorem}\label{theorem_uniqueness}
	Equation~\eqref{main_PDE} has exactly one solution with conditions $u(0)=0$ and $0\le u(x) \le 1$ which is given by~\eqref{explicit_solution}.
\end{theorem}

\begin{proof}

Note that
\[ r(x,0) \le u(x) \le s(x,0) \]
We wish to show that for any $t\ge 0$,  $r(x,t) \le u(x) \le s(x,t)$. We do this using the maximum principle.
First, notice that since $u$ satisfies~\eqref{main_PDE}, we have $\lim_{x\to\infty} u(x) = 1$, and also $\lim_{x\to\infty} r(x,t) = 1$.
Let $v(x,t)=u(x)-r(x,t)$. Suppose that there exists $(x_0,t_0)$ such that $v(x,t)<0$. Hence $v$ attains its minimum on $[0,\infty)\times [0,t_0]$ at some interior point and there we have $v(x_1,t_1)<0$ which implies $u(x_1,t_1)< r(x_1,t_1)$ and therefore $(u(x_1,t_1)-1)^2 > (r(x_1,t_1)-1)^2$. Therefore $L v(x_1,t_1)>0$ where $L=\partial_t-\partial_x^2$. Hence by maximum principle (\cite{Evans} Theorem 11.page 375) $v$ is constant which is impossible. Hence we have proved $r(x,t) \le u(x)$. A similar argument implies $u(x)\le s(x,t)$.

On the other hand, it is clear by definition of $r$ and $s$ and the fact that the critical branching processes extinct with probability 1, that as $t\to\infty$,
\[ r(x,t) \to p(x) \quad s(x,t)\to p(x) \]
which implies $u(x)=p(x)$.

\end{proof}

\section{Non-negative Solutions}\label{section: non-negative}
In this section we study the non-negative solutions to~\eqref{main_PDE}.

Consider a process $Z(t)$ which has a minor difference with $X(t)$ in that when the particles reach the origin, they stop there forever. Let $f:\mathbb{Z}\to \mathbb{R}$ be a function and define
\[ q^f(x,t) := \mathbb{E}^{\{x\}} (\prod_{y\in Z(t)} f(y)) \]
which by ${E}^{\{x\}}$ we mean that we are starting with a particle at $x$.

We may omit the superscript $f$ in case that there is no ambiguity. By an argument similar to~\cite{Grimmett}, page 522, we find that $q$ satisfies the equation~\eqref{main_evolution_PDE} and we have $q(0,t)=0$.

\begin{proposition}
For any non-negative solution $u(x)$ of equation~\eqref{main_evolution_PDE}, we have
\[ p(x) \le u(x) \]
\end{proposition}

\begin{proof}
It is obvious that $u$ is a stationary solution of equation~\eqref{main_evolution_PDE}, i.e. $q^u(x,t)=u(x)$. On the other hand
\[ q^u(x,t) := \mathbb{E}^{\{x\}} (\prod_{y\in Z(t)} u(y)) \]
\[ = \mathbb{E}^{\{x\}} (\prod_{y\in Z(t)} u(y) 1_{\text{No particle reaches $0$ until time $t$}})\]
Since the population extinct with probability 1, hence the expression inside the expectation converges almost surely to the characteristic function of the event that no particle reaches $0$ forever. Therefore by the Fatou's lemma,
\[ \liminf_{t\to\infty} q^u(x,t) \ge \mathbb{P}^{\{x\}} (\text{No particle reaches $0$} ) = p(x) \]
hence the proof is complete.
\end{proof}

Hence,
\begin{corollary}\label{corollary:non-negative}
$p(x)$ is the least non-negative solution of~\eqref{main_PDE}.
\end{corollary}

\section{General Differential Equations}\label{section: general}

\begin{definition}
By a \textit{polynomial parabolic equation} we mean a differential equation in the form
\[ \frac{\partial u}{\partial t} = \frac{1}{2}\frac{\partial^2 u}{\partial x^2} + F(u) \]
where $F$ is a polynomial with real coefficients. If we have $F(u)=\lambda(G(u)-u)$ where $G$ has non-negative coefficients then the polynomial parabolic equation is called \textit{positive} and if moreover the sum of the coefficients of $G$ is 1, it is called \textit{probabilistic}.
\end{definition}

\begin{remark}
	 Many positive polynomial parabolic equations can be transformed to probabilistic polynomial parabolic equations by linear transformations.
\end{remark}

We wish to study the following probabilistic polynomial parabolic equation with initial condition:
\begin{eqnarray}\label{eq:general_PDE}
	\frac{\partial u}{\partial t} = \frac{1}{2}\frac{\partial^2 u}{\partial x^2} u + \lambda(G(u)-u), x\ge 0\\
	u(0,t)=0.
\end{eqnarray}

We distinguish three cases:

\begin{definition}
A probabilistic polynomial parabolic equation is called critical, sub-critical, super-critical, respectively if $G^\prime(0)$ is equal to, less than or more than 1.
\end{definition}

To any probabilistic polynomial recurrence relation there corresponds a spatial branching process in the following way:

Let a particle be at the origin on the real line and assume it performs a standard Brownian motion until a random time which is exponentially distributed with parameter $\lambda$ and at that time it dies and with probability $a_n$ gives birth to $n$ particles same as itself, where $a_n$ is the coefficient of $x^n$ in $G$. Let $X(t)$ be the (multi-)set of particles at time $t$. Let $p(x)$ be the probability that no one of the particles reaches $x$. By a similar argument as Theorem \ref{theorem_uniqueness} we find,

\begin{theorem}\label{theorem_uniqueness_general}
	In the critical and sub-critical cases, equation~\eqref{eq:general_PDE} has exactly one solution with conditions $u(0)=0$ and $0\le u(x) \le 1$ which is $p(x)$.
\end{theorem}

\begin{remark}
In the super-critical case the story is a bit different. The same argument as Theorem \ref{theorem_uniqueness} implies that any solution $0\le u(x) \le 1$ satisfies,
\[ \mathbb{P}(\text{The population extincts and never reaches $x$}) \le u(x) \le p(x)\]
\end{remark}

A similar argument as Corollary~\ref{corollary:non-negative} shows that

\begin{corollary}\label{corollary:non-negative_general}
	In the critical and sub-critical cases, $p(x)$ is the least non-negative solution of~\eqref{eq:general_PDE}.
\end{corollary}

\end{document}